\documentclass[12pt,a4paper]{article}

\usepackage[T1]{fontenc}
\usepackage[utf8]{inputenc}

\usepackage{textcomp}
\usepackage[official]{eurosym}
\usepackage{amssymb,mathrsfs}
\usepackage{amsthm}
\usepackage{mathtools}
\usepackage{latexsym}

\newtheorem{theorem}{Theorem}[section]

\newtheorem{proposition}[theorem]{Proposition}

\newtheorem{question}{Question}

\newtheorem{example}{Example}

\begin{document}
	\title{A weakly Lindel\"of space which does not have the property $^*\mathcal{U}_{fin}(\mathcal{O},\mathcal{O})$}
	\author{ Servet Soyarslan\footnote{Corresponding author. \newline E-mail addresses:  servet.soyarslan@gmail.com (S. Soyarslan), osul@metu.edu.tr (S. \"Onal). },  S\"uleyman \"Onal$^{(a)}$,
	}
	
	\maketitle	
	
	{\scriptsize a. Middle East Technical University, Department of Mathematics, 06531 Ankara, Turkey.} 
		
\begin{abstract}
We show that there exists a Tychonoff weakly Lindelöf space  which does not have the property $^*\mathcal{U}_{fin}(\mathcal{O},\mathcal{O}).$ This result answers the following open questions. \cite{singh2021further} Does a weakly Lindelöf space have  the property $^*\mathcal{U}_{1}(\mathcal{O},\mathcal{O})$? \cite{Song} Is there a Tychonoff 1-star-Lindel\"of space which does not have the property $^*\mathcal{U}_{fin}(\mathcal{O},\mathcal{O})$?
\end{abstract}

{\scriptsize \textbf{Keywords}: Selection principle, Lindel\"of, Star-covering, Stone–\v{C}ech compactification.}
	
	{\scriptsize\textbf{MSC:} 54D35,	54D20.}

\section{Introduction}

In this paper, $\mathbb{N}=\{0,1,2,\ldots\}$. $cl_{X}(A)$ denotes the closure of $A$ in the space $X$. $\omega$ denotes the first infinite cardinal. $\beta X$ denotes the  Stone–\v{C}ech compactification of the space $X$. $|A|$ denotes the cardinality of the set $A$.

Let $A$ be a subset of a space $X$ and $\mathcal{U}$ a family of subsets of $X$. Then $St(A,\mathcal{U})=\bigcup\{U\in \mathcal{U}: U\cap A\neq \emptyset\}.$

The following definitions and some references about this topic can  be found in \cite{singh2021further} and \cite{Song}.

$\mathcal{O}$ denotes the family of all open covers of a space $X$.

A space X is said to have the property $^*\mathcal{U}_{1}(\mathcal{O},\mathcal{O})$ if for each sequence $(\mathcal{U}_n:n \in \omega)$ of elements of $\mathcal{O}$, there is a sequence $(U_n:n \in \omega)$ such that for each $n\in \omega$,  $U_n \in \mathcal{U}_n$ and 
$\{St(\bigcup_{i\in \omega} U_i, \mathcal{U}_n) :n\in \omega\} \in \mathcal{O}$.
	
A space X is said to have the $^*\mathcal{U}_{fin}(\mathcal{O},\mathcal{O})$ if for each sequence  $(\mathcal{U}_n:n \in \omega)$ of elements of $\mathcal{O}$, there is a sequence $(\mathcal{V}_n:n\in \omega)$ such that for each $n\in \mathbb{N}$, $\mathcal{V}_n$ is a finite subset of $ \mathcal{U}_n$ and 
$\{St(\bigcup_{i\in \omega} (\cup \mathcal{V}_i), \mathcal{U}_n) :n\in \omega\} \in \mathcal{O}$.

A space $X$ is 1-star-Lindel\"of if for every open cover $\mathcal{U}$ of $X$, there exists a countable subfamily $\mathcal{V}$ of $\mathcal{U}$ such that $St(\cup \mathcal{V}, \mathcal{U}) =X$.

A space $X$ is  weakly Lindel\"of if for every open cover $\mathcal{U}$ of $X$, there exists a countable subfamily $\mathcal{V}$ of $\mathcal{U}$ such that $cl_X(\cup \mathcal{V}) =X$.

In the next section, we give a  positive answer to the following questions.

\begin{question}	(Question 2.8. in \cite{Song}) Is there a Tychonoff 1-star-Lindel\"of space which does not have the property $^*\mathcal{U}_{fin}(\mathcal{O},\mathcal{O})$? 
\end{question}

\begin{question} (Question 4.14. in \cite{singh2021further})
Does a weakly Lindelöf space have the property $^*\mathcal{U}_{1}(\mathcal{O},\mathcal{O})$?

\end{question}

\section{An Example}

Clearly,  if $X$ does not have the property $^*\mathcal{U}_{fin}(\mathcal{O},\mathcal{O})$, then  $X$ does not have the property $^*\mathcal{U}_{1}(\mathcal{O},\mathcal{O})$.
Clearly, if $X$ is weakly Lindel\"of, then $X$ is 1-star-Lindel\"of. Thus, if we give an example of a Tychonoff weakly Lindelöf space $X$ which does not have the property $^*\mathcal{U}_{fin}(\mathcal{O},\mathcal{O})$, then we  answers both Questions 1 and 2.

 First, we give the following  known fact. Then, we give that example. 
\begin{proposition}\label{Pro}
	Let us consider the discrete topological space $\mathbb{N}^\mathbb{N}$ and let $M,N \subseteq \mathbb{N}^\mathbb{N}$. Then, we have the followings. 
	
	(1) If $M\cap N  = \emptyset$, then $cl_{\beta \mathbb{N}^\mathbb{N}}(M)\cap cl_{\beta \mathbb{N}^\mathbb{N}}(N) = \emptyset$.

 (2) If $|M\cap N|<\omega$, then $cl_{\beta \mathbb{N}^\mathbb{N}}(M)\cap cl_{\beta \mathbb{N}^\mathbb{N}}(N) -\mathbb{N}^\mathbb{N}=\emptyset$.
 
 (3) Let $A$ be any subset of $\mathbb{N}^\mathbb{N}$. Then, $cl_{\beta \mathbb{N}^\mathbb{N}}(A)$ is a clopen subset  of ${\beta \mathbb{N}^\mathbb{N}}.$

\end{proposition}
\begin{proof} (1) Because $\mathbb{N}^\mathbb{N}$ is discrete, $M$ and $N$ are closed  in $\mathbb{N}^\mathbb{N}$. Because  $M\cap N  = \emptyset$ and $\mathbb{N}^\mathbb{N}$ is normal,
 from Corollary 3.6.4  in \cite{engelking1989general}, $cl_{\beta \mathbb{N}^\mathbb{N}}(M)\cap cl_{\beta \mathbb{N}^\mathbb{N}}(N) = \emptyset$.

(2)	$cl_{\beta \mathbb{N}^\mathbb{N}}(M)=cl_{\beta \mathbb{N}^\mathbb{N}}((M-(M\cap N))\cup (M\cap N))=cl_{\beta \mathbb{N}^\mathbb{N}}(M-(M\cap N))\cup cl_{\beta \mathbb{N}^\mathbb{N}}(M\cap N)$. 
Because $M\cap N$ is finite, $M\cap N$ is closed in ${\beta \mathbb{N}}$. Thus,
	 $cl_{\beta \mathbb{N}^\mathbb{N}}(M\cap N)=M\cap N$. Therefore, 
	$cl_{\beta \mathbb{N}^\mathbb{N}}(M)=cl_{\beta \mathbb{N}^\mathbb{N}}(M-(M\cap N))\cup (M\cap N)$. Hence $cl_{\beta \mathbb{N}^\mathbb{N}}(M)\cap cl_{\beta \mathbb{N}^\mathbb{N}}(N)=\big(cl_{\beta \mathbb{N}^\mathbb{N}}(M-(M\cap N))\cap cl_{\beta \mathbb{N}^\mathbb{N}}(N)\big) \cup (M\cap N \cap cl_{\beta \mathbb{N}^\mathbb{N}}(N) )$. From (1), $cl_{\beta \mathbb{N}^\mathbb{N}}(M-(M\cap N))\cap cl_{\beta \mathbb{N}^\mathbb{N}}(N)=\emptyset$. Therefore $cl_{\beta \mathbb{N}^\mathbb{N}}(M)\cap cl_{\beta \mathbb{N}^\mathbb{N}}(N)=M\cap N \cap cl_{\beta \mathbb{N}^\mathbb{N}}(N)$. Thus, $cl_{\beta \mathbb{N}^\mathbb{N}}(M)\cap cl_{\beta \mathbb{N}^\mathbb{N}}(N) -\mathbb{N}^\mathbb{N}=\emptyset$.

(3) This is obvious from Corollary 3.6.5  in \cite{engelking1989general}.

\end{proof}

\begin{example}

There is a Tychonoff weakly Lindelöf space $X$ which does not have the property $^*\mathcal{U}_{fin}(\mathcal{O},\mathcal{O}).$
	
\end{example}

\begin{proof}

We put the discrete topology on $\mathbb{N}^\mathbb{N}$ and consider the Stone–\v{C}ech compactification $\beta\mathbb{N}^\mathbb{N}$.

Throughout this proof, for the sake of shortness,  $\overline{A}$ and $A^c$ denote the closure and complement of $A$ in the space $\beta\mathbb{N}^\mathbb{N}$, respectively (i.e. $\overline{A}=cl_{\beta\mathbb{N}^\mathbb{N}}(A)$  and $A^c=\beta\mathbb{N}^\mathbb{N}\setminus A$).

 For any $i_0,i_1,\ldots,i_k\in \mathbb{N}$,
let us define  $$A(i_0,i_1,\ldots, i_k)=\{x\in \mathbb{N}^\mathbb{N}: x_0=i_0,x_1=i_1,\ldots, x_k=i_k\}.$$ 
Then, 
define  $$B(k)=\bigcup \big\{\overline{A(i_0,i_1,\ldots,i_k)}:i_0,\ldots,i_k\in \mathbb{N}\big\}.$$ Now, set $$Y=\bigcap_{k\in \mathbb{N}} B(k).$$

Let $<$ be the  product order on $\mathbb{N}^\mathbb{N}$, i.e, $x<y$ iff $x_i<y_i$ for all $i\in \mathbb{N}$. 

Say $|\mathbb{N}^\mathbb{N}|=\kappa$. Then, we can write $\mathbb{N}^\mathbb{N}=\{j_i:i\in \kappa\}$ such that $j_l\neq j_t$ if $l\neq t$.

  \textbf{Claim 1.}  Take any $M\subseteq \mathbb{N}^\mathbb{N}$ and $x\in \mathbb{N}^\mathbb{N}$. If $|M|< \kappa$, then there exists a subset $\{y_i:i\in \omega\}\subseteq \mathbb{N}^\mathbb{N}-M$ such that $x<y_0<y_1<y_2<\ldots$.

\textbf{Proof of  Claim 1.} For any $z=(z_n)_n\in\mathbb{N}^\mathbb{N}$, set $N_z=\{y\in \mathbb{N}^\mathbb{N}: z<y\}$.   For any $z_n$ of  $z=(z_n)_n$  if we set $P_n=\{z_n+m:m\in \mathbb{N}-\{0\}\}$, then we get $N_z=\Pi_{n\in \omega} P_n$ and $|P_n|=\omega$. So, $|N_z|=|\Pi_{n\in \omega} P_n|=\kappa$. Therefore, $|N_z-M|=\kappa$ because  $|M|< \kappa$. 

We use induction. Suppose we get $y_0,\ldots ,y_n\in N_x-M$ such that $x<y_0<y_1<\ldots<y_n$. Because  $|N_{y_n}-M|=\kappa$, we can choose a $y_{n+1}\in N_{y_n}-M$. Thus, from the definition of $N_{y_n}$, we get $y_0,\ldots ,y_n, y_{n+1}\in N_x-M$ and   $x<y_0<y_1<\ldots<y_n<y_{n+1}$. From the induction, proof of Claim 1 is completed.

Now,  for all $i\in \kappa$ by using transfinite induction, we will define a  $D_{j_i}\subseteq \mathbb{N}^\mathbb{N}$. First,  for $j_0\in \mathbb{N}^\mathbb{N}$ we choose $\omega$ many elements of $\mathbb{N}^\mathbb{N}$
 which have the condition that $j_0<x_0<x_1<\ldots$ and define $D_{j_0}=\{x_k:k\in \omega\}\subseteq \mathbb{N}^\mathbb{N}$. Then, for any $\alpha\in \kappa$ suppose  we get $D_{j_i}\subseteq \mathbb{N}^\mathbb{N}$ with $|D_{j_i}|=\omega$ for all
$0\leq i<\alpha$. After that, for ${j_\alpha}\in \mathbb{N}^\mathbb{N}$ we choose $\omega$ many elements of $\mathbb{N}^\mathbb{N}-\big((\bigcup_{i<\alpha}D_{j_i})\cup \{j_i:i\in \alpha\}\big)$
which have the condition that $j_\alpha<x_0<x_1<\ldots$ and define $D_{j_\alpha}=\{x_k:k\in \omega\}\subseteq \mathbb{N}^\mathbb{N}$. Note that we can get these $D_{j_i}$ for all $i\in \kappa$ because  of Claim 1. So,  we get a disjoint family $\{D_{j_i}\subseteq \mathbb{N}^\mathbb{N}: i\in \kappa\}$. Now,
 define $$X=Y\cup (\bigcup_{i\in \kappa}\overline{D_{j_i}}).$$
 
The topology of the space $X$ is the subspace topology inherited from $\beta\mathbb{N}^\mathbb{N}$. We will see that this space $X$ is weakly Lindelöf and does not have the property  $^*\mathcal{U}_{fin}(\mathcal{O},\mathcal{O})$.

 \textbf{Claim 2.} $X$ is a weakly Lindel\"of space .

 \textbf{Proof of  Claim 2.} Let $\mathcal{U}$ be an open cover of $X$. 
 
 Because $$Y=\bigcap_{k\in \mathbb{N}} B(k)=\bigcap_{k\in \mathbb{N}}(\bigcup \big\{\overline{A(i_0,\ldots,i_k)}:i_0,\ldots,i_k\in \mathbb{N} \big\}),$$
then by using distributivity and eliminating empty sets from (1) of Proposition \ref{Pro}, we can get

\begin{equation}\label{1}
 Y=\bigcup \big\{\overline{A(x_0)}\cap \overline{{A(x_0,x_1)}}\cap \overline{{A(x_0,x_1,x_2)}}\cap \ldots: x=(x_i)_i\in \mathbb{N}^\mathbb{N} \big \}.
\end{equation}

Now, set $$\mathcal{N}=\{\overline{A(x_0)}\cap \overline{{A(x_0,x_1)}}\cap \overline{{A(x_0,x_1,x_2)}}\cap \ldots:x=(x_i)_i\in \mathbb{N}^\mathbb{N}\}$$ and 

$$\mathcal{M}=\{\mathcal{G}\subseteq \{\overline{A(i_0,\ldots,i_k}):i_0,\ldots,i_k, k\in \mathbb{N} \}:|\mathcal{G}|<\omega\}.$$

Take any $N\in \mathcal{N}$. Note that $N$  is a compact subspace of $\beta\mathbb{N}^\mathbb{N}$ because $N$ is closed subspace of compact space $\beta\mathbb{N}^\mathbb{N}$. Because $N\subseteq Y \subseteq (\cup \mathcal{U})$ and $N$ is compact, we can choose a finite subfamily $\mathcal{U}_N$ of the cover  $\mathcal{U}$ of $X$ such that $N\subseteq \cup \mathcal{U}_N.$ Because $\cup \mathcal{U}_N$ is open in $X$, we can choose a  $(\cup \mathcal{U}_N)'$ open subset of $\beta\mathbb{N}^\mathbb{N}$ such that $X\cap(\cup \mathcal{U}_N)'=\cup \mathcal{U}_N$. Now, by using these chosen sets,  we can give the following subclaim.

\textbf{Subclaim 1. }  For any $N\in \mathcal{N}$, there exists a $\mathcal{G}_N\in \mathcal{M}$ such that $N\subseteq (\cap\mathcal{G}_N)$ and  $\cap\mathcal{G}_N\subseteq (\cup \mathcal{U}_N)'$.

\textbf{Proof of  Subclaim 1.} Suppose $N=\overline{A(x_0)}\cap \overline{{A(x_0,x_1)}}\cap \overline{{A(x_0,x_1,x_2)}}\cap \ldots$, for a fixed  $x=(x_i)_i\in \mathbb{N}^\mathbb{N}.$ Then, $$\mathcal{H}=\{(\cup \mathcal{U}_N)',  \big(\overline{A(x_0))}\big)^c,  \big(\overline{A(x_0, x_1)}\big)^c,  \big(\overline{{A(x_0,x_1,x_2)}}\big)^c, \ldots \}$$ is an open cover of $\beta\mathbb{N}^\mathbb{N}$. So, there is a finite subcover $\mathcal{K}$ of $\mathcal{H}$ because $\beta\mathbb{N}^\mathbb{N}$ is compact. If $\mathcal{K}=\{(\cup \mathcal{U}_N)'\}$, then we can set $\mathcal{G}_N=\{\overline{A(x_0)}\}$. If $\mathcal{K}\neq \{(\cup \mathcal{U}_N)'\}$, then we set $\mathcal{G}_N=\{U: (U)^c\in (\mathcal{K}-\{(\cup \mathcal{U}_N)'\})\}$. Note that the elements of $\mathcal{G}_N$ have the forms $\overline{A(x_0, \ldots ,x_k)}$. Then, $N\subseteq (\cap\mathcal{G}_N)$ because of  definition of $N$ and $\mathcal{G}_N$. In addition, for any $x\in\cap\mathcal{G}_N$, there exists a $K\in \mathcal{K}$ such that $x\in K$ because $\bigcup \mathcal{K}=\beta\mathbb{N}^\mathbb{N}$. Because of definition of $\mathcal{G}_N$, $x\in K=(\cup \mathcal{U}_N)'$. Thus $\cap\mathcal{G}_N\subseteq (\cup \mathcal{U}_N)'$. Hence $N\subseteq (\cap\mathcal{G}_N) \subseteq (\cup \mathcal{U}_N)'.$  The proof of Subclaim 1 is completed.

Note that $|\mathcal{M}|=\omega$ because $|\{\overline{A(i_0,\ldots,i_k)}:i_0,\ldots,i_k, k\in \mathbb{N} \}|=\omega$. Now take any $\mathcal{G}\in \mathcal{M}$. If  $\{\mathcal{L}\subseteq {\mathcal{U}}:\mathcal{L}$ is finite and $\cap \mathcal{G}\subseteq (\cup\mathcal{L})'\}\neq \emptyset$, then  choose  and fix an $\mathcal{U}_\mathcal{G}\in \{\mathcal{L}\subseteq {\mathcal{U}}:\mathcal{L}$ is finite and $\cap \mathcal{G}\subseteq (\cup\mathcal{L})'\}$ for this $\mathcal{G}\in \mathcal{M}$. Now define $\mathcal{M}^*=\{\mathcal{G}\in \mathcal{M}: \mathcal{U}_\mathcal{G}$ can be chosen for $\mathcal{G}\}$. From Subclaim 1, $\mathcal{M}^*\neq \emptyset$. Then, define

 $$\mathcal{V}=\bigcup\{\mathcal{U}_{\mathcal{G}}\subseteq \mathcal{U}:  \mathcal{G}\in \mathcal{M}^* \}.$$ Then, $|\mathcal{V}|\leq \omega$ because each $\mathcal{U}_{\mathcal{G}}$ is finite and  $|\mathcal{M}^*|\leq \omega$. 
 	
 By using the fact $\cap \mathcal{G}\subseteq (\cup \mathcal{U}_{\mathcal{G}})'$  and  Subclaim 1, we get that for any $N\in \mathcal{N}$ there exists a 
 $\mathcal{G} \in \mathcal{M}^*$ such that
  $N\subseteq (\cap\mathcal{G}) \subseteq (\cup \mathcal{U}_{\mathcal{G}})'$. Thus, $N\subseteq  (\cup \mathcal{U}_{\mathcal{G}})'\cap X =(\cup \mathcal{U}_{\mathcal{G}})$ for any $N\in \mathcal{N}$. Therefore, by considering (\ref{1}),
 	 $Y=\bigcup \mathcal{N}\subseteq  (\cup(\bigcup\{\mathcal{U}_{\mathcal{G}}\subseteq \mathcal{U}:  \mathcal{G}\in \mathcal{M}^*	\}))=\bigcup \mathcal{V}.$ Thus, $Y\subseteq \bigcup\mathcal{V}$ and   $\mathcal{V}$ is a countable subfamily of $\mathcal{U}$.

 On the other hand $\mathbb{N}^\mathbb{N}\subseteq Y$ because $x\in A(x_0,x_1,\ldots,x_i)\subseteq B(i)$ for any $x\in \mathbb{N}^\mathbb{N}$ and any $i\in \mathbb{N}$.
 
 Hence $\beta\mathbb{N}^\mathbb{N}=\overline{\mathbb{N}^\mathbb{N}}\subseteq \overline{Y}\subseteq \overline{\bigcup\mathcal{V}}\subseteq {\beta\mathbb{N}^\mathbb{N}}.$ Then $\beta\mathbb{N}^\mathbb{N}= \overline{\bigcup\mathcal{V}}.$
 Thus, $cl_{X}(\bigcup\mathcal{V})=X$. Hence, $X$ is weakly Lindel\"of because $\mathcal{V}$ is a countable subcover of the cover $\mathcal{U}$ of $X$. Proof of Claim 2 is completed.
 
 If we prove the following claim the proof will be completed.

  \textbf{Claim 3.} $X$ does not have the property  $^*\mathcal{U}_{fin}(\mathcal{O},\mathcal{O})$.

 \textbf{Proof of  Claim 3.} For any $k\in \omega$ let us define $$\mathcal{U}_k=\{\overline{{A(i_0,i_1,\ldots, i_k)}}\cap X : i_0,\ldots ,i_k\in \mathbb{N}\}\cup \{\overline{D_{j_i}}:i\in \kappa\}.$$  
 
 From $(3)$ of  Proposition \ref{Pro}, $\overline{{A(i_0,i_1,\ldots, i_k)}}$ and  $\overline{D_{j_i}}$ are open in ${\beta\mathbb{N}^\mathbb{N}}.$ Thus, it is easy to see that $\mathcal{U}_k$ is an open cover of $X$ for any $k\in \mathbb{N}$. Now, define the sequence $(\mathcal{U}_k:k\in \omega)$ of  open covers of $X$. Take any sequence $(\mathcal{V}_k:k\in \omega)$ such that $\mathcal{V}_k$ is a finite subset of $\mathcal{U}_k$. 
 
 Now, we will see that the family  $\{st(\cup( \bigcup_{i\in \omega  }\mathcal{V}_i), \mathcal{U}_k): k\in \omega\}$ is 
 not an open cover of $X$. To see it, by supposing that for any $k\in \mathbb{N}$ there exists $n(k)\in \mathbb{N}$ 
 such that $$\mathcal{A}_k=\{\overline{A(i^{t}_0,i^{t}_1,\ldots,i^{t}_k)}\cap X : 0\leq t\leq n(k)\}\subseteq  \mathcal{U}_k,$$ 
   and without loss of generality we may assume that $$\bigcup_{i\in \omega  }\mathcal{V}_i=(\bigcup\{\mathcal{A}_k:k\in \omega\})\cup (\{\overline{D_{j_{i_s}}}:s\in \omega\}).$$
 
 Now we will define a special point $y^*=(y_k )_k\in \mathbb{N}^\mathbb{N}$ by defining each component $y_k$ of $y^*$.
  Let $e_{s,l}$ denotes the $l$-th component of $j_{i_s}$, i.e., $e_{s,l}=n_l$ if $j_{i_s}=(n_0,n_1,n_2,\ldots)\in\mathbb{N}^\mathbb{N}$.  
  Now,  let $A_k$ be the set of   indexes of the elements of $\mathcal{A}_k$, i.e., $A_k=\{i^{t}_r: 0\leq r\leq k,$   
  and $  0\leq t\leq n(k)
  \}$.  
  Now, define $y_k=1+\max(A_k\cup\{e_{k,k}\})$ for all $k\in \mathbb{N}$ and so we get the special point $y^*=(y_k )_k$ of  $\mathbb{N}^\mathbb{N}.$ Thus, there exists an $i^*\in \kappa$ such that $j_{i^*}=y^*$. Therefore, $D_{j_{i^*}}=D_{y^*}\in \{D_{j_i}:i\in \kappa\}$.  

\textbf{Subclaim 2. } We have the following.

(i) $(\cup(\bigcup_{i\in \omega}\mathcal{V}_i))\cap \overline{D_{y^*}}=\emptyset$;

(ii) $\overline{D_{y^*}}-\mathbb{N}^\mathbb{N}\neq \emptyset$;

(iii) For any $x\in \overline{D_{y^*}}-\mathbb{N}^\mathbb{N}$ and
 for each $k\in \mathbb{N}$, $x\notin st(\cup(\bigcup_{i\in \omega}\mathcal{V}_i), \mathcal{U}_k).$

\textbf{Proof of  Subclaim 2.}

(i) Take any $d=(d_k)_k \in D_{y^*}$. So, for any $k\in \mathbb{N}$, we know that $d_k>y_k>i^{t}_k$ where $0\leq t\leq n(k)$. Thus, for any $k\in \mathbb{N}$ and $0\leq t\leq n(k)$, we get $d\notin A(i^t_0,\ldots,i^t_k)$. So, $D_{y^*}\cap A(i^t_0,\ldots,i^t_k)=\emptyset$. Thus, from (1) 
of  Proposition \ref{Pro}, $\overline{D_{y^*}}\cap \overline{A(i^t_0,\ldots,i^t_k)}=\emptyset$ for any $k\in \mathbb{N}$ and $0\leq t\leq n(k)$.
Therefore $\overline{D_{y^*}}\cap(\cup(\bigcup\{\mathcal{A}_k:k\in \omega\}))=\emptyset$. 

On the other hand,  $y_s>e_{s,s}$  for any $s\in \omega$. Thus, $y^*\neq j_{i_s}$ for any $s\in \omega$. So, $D_{y^*}\notin \{D_{j_{i_s}}:s\in \omega\}$. Therefore, $D_{y^*}\cap D_{j_{i_s}}=\emptyset$ for any $s\in \omega$ because the family $\{D_{j_i}:i\in \kappa\}$ is disjoint. Thus, from (1) of  Proposition \ref{Pro}, $\overline{D_{y^*}}\cap \overline{D_{j_{i_s}}}=\emptyset$ for any $s\in \omega$.
Hence $\overline{D_{y^*}}\cap(\bigcup\{\overline{D_{j_{i_s}}}:s\in\omega\})=\emptyset$.

  Thus $(\cup(\bigcup_{i\in \omega}\mathcal{V}_i))\cap \overline{D_{y^*}}=\emptyset$.
  
  (ii) Assume $\overline{D_{y^*}}-\mathbb{N}^\mathbb{N}= \emptyset$. Then $\overline{D_{y^*}}\subseteq\mathbb{N}^\mathbb{N}$. Thus, from  (3) of Proposition \ref{Pro},  $\{\{x\}:x\in\overline{D_{y^*}}\}$ is an open cover of the compact set $\overline{D_{y^*}}$. But $\{\{x\}:x\in\overline{D_{y^*}}\}$  has no finite subcover. This is a contradiction.
  
  (iii) Take any $x\in \overline{D_{y^*}}-\mathbb{N}^\mathbb{N}$ and a $k\in \mathbb{N}$. Now, take any $A\in \mathcal{U}_k$ which has the condition that  $A\cap (\cup(\bigcup_{i\in \omega}\mathcal{V}_i))\neq\emptyset.$ To see $x\notin st(\cup(\bigcup_{i\in \omega}\mathcal{V}_i), \mathcal{U}_k)$, we will see that $x\notin A$. From the definition of   
  $ \mathcal{U}_k$, $A\in  \{\overline{{A(i_0,i_1,\ldots, i_k)}}\cap X : i_0,\ldots,i_k\in \mathbb{N}\}$ or $A\in \{\overline{D_{j_i}}:i\in \kappa\}$. Suppose $A=  \overline{{A(i_0,i_1,\ldots, i_k)}}\cap X$ for some $i_0,\ldots,i_k\in \mathbb{N}$. Then $|{A(i_0,i_1,\ldots, i_k)}\cap D_{y^*}|\leq 1$ because elements of $D_{y^*}$ are $<$-increasing. Therefore, from (2) of Proposition \ref{Pro},
   $\overline{A(i_0,i_1,\ldots, i_k)}\cap \overline{D_{y^*}}-\mathbb{N}^\mathbb{N}=\emptyset$. Thus $A\cap \overline{D_{y^*}}-\mathbb{N}^\mathbb{N}=\emptyset$. Therefore, $x\notin A$ because $x\in \overline{D_{y^*}}-\mathbb{N}^\mathbb{N}$. Now, suppose $A\in \{\overline{D_{j_i}}:i\in \kappa\}$. From (i) of this subclaim, $A\neq \overline{D_{y^*}}$ because $A\cap (\cup(\bigcup_{i\in \omega}\mathcal{V}_i))\neq\emptyset$. Thus, there is a $j_i\in \mathbb{N}^\mathbb{N}$ such that  $y^*\neq j_i$ and $A=\overline{D_{j_i}}$. 
   Because $y^*\neq j_i$, $D_{j_i}\cap D_{y^*}=\emptyset$. So, from (1) of Proposition \ref{Pro}, $A\cap \overline{D_{y^*}}=\overline{D_{j_i}}\cap \overline{D_{y^*}}=\emptyset$. Hence $x\notin A$. The proof of Subclaim 2 is completed.
  
  From Subclaim 2, $\{st(\cup(\bigcup_{i\in \omega}\mathcal{V}_i), \mathcal{U}_k):k\in \omega\}$ is not an open cover of $X$. (Note that $\cup(\bigcup_{i\in \omega}\mathcal{V}_i)=\bigcup_{i\in \omega}(\cup\mathcal{V}_i)$.) Hence, $X$ is not  $^*\mathcal{U}_{fin}(\mathcal{O},\mathcal{O})$.
    
\end{proof}

\bibliographystyle{plain}



\end{document}